\documentclass{amsart}

\usepackage{amsmath}
\usepackage{amssymb}
\usepackage[all]{xy}
\usepackage{picinpar}

\def\eu{\mathfrak}
\def\ma{\mathbb}
\def\s{\bigskip}

\def\lam#1{k(\Lambda_{P^#1})}

\def\F{{\ma F}_q}
\def\G#1{\big(R_T/(#1)\big)^{\ast}}
\def\p{{\eu p}_{\infty}}

\def\d{\displaystyle}
\def\Witt#1{\stackrel{_{\bullet}}{#1}}

\newcommand{\Id}{\operatorname{Id}}
\newcommand{\Gal}{\operatorname{Gal}}
\newcommand{\im}{\operatorname{im}}
\newcommand{\Aut}{\operatorname{Aut}}
\newcommand{\lcm}{\operatorname{lcm}}

\newcommand{\integer}{\genfrac[]{0.5pt}0}
\newcommand{\integerchico}{\genfrac[]{0.5pt}1}
\newcommand{\mayor}{\genfrac\lceil\rceil{0.5pt}0}
\newcommand{\mayorchico}{\genfrac\lceil\rceil{0.5pt}1}
\def\proof{\noindent{\sc Proof}:\ }

\newcounter{bean}

\def\l{
\begin{list}
{\rm{(\alph{bean})}}{\usecounter{bean}
\setlength{\labelwidth}{0.8in}
\setlength{\labelsep}{0.3cm}
\setlength{\leftmargin}{1cm}}}

\numberwithin{equation}{section}
\newtheorem{theorem}{Theorem}[section]

\newtheorem{proposition}[theorem]{Proposition}
\newtheorem{lemma}[theorem]{Lemma}
\newtheorem{corollary}[theorem]{Corollary}

\theoremstyle{remark}
\newtheorem{remark}[theorem]{Remark}

\title[Kronecker--Weber Theorem in function fields]
{A combinatorial proof of the Kronecker--Weber
Theorem in positive characteristic}

\author[J.C. Salas]{Julio Cesar Salas--Torres}
\address{Universidad Aut\'onoma de la Ciudad de M\'exico\\
Academia de Matem\'aticas. Plantel San Lorenzo Tezonco\\
Prolongaci\'on San Isidro No. 151 Col. San Lorenzo,
Iztapalapa, C.P. 09790, M\'exico, D.F.}
\email{jcstorres88@hotmail.com, torres1jcesar0@gmail.com}

\author[M. Rzedowski]{Martha Rzedowski--Calder\'on}
\address{Departamento de Control Autom\'atico\\
Centro de Investigaci\'on y de Estudios Avanzados del I.P.N.}
\email{mrzedowski@ctrl.cinvestav.mx}

\author[G. Villa]
{Gabriel Villa--Salvador}
\address{
Departamento de Control Autom\'atico\\
Centro de Investigaci\'on y de Estudios Avanzados del I.P.N.}
\email{gvillasalvador@gmail.com, gvilla@ctrl.cinvestav.mx}

\subjclass[2010]{Primary 11R60; Secondary 11R18, 11R32, 11R58}

\keywords{Maximal abelian extension, congruence function fields,
global fields, Kronecker--Weber Theorem, cyclotomic function fields,
Artin--Schreier extensions, Witt vectors.}

\date{July 11, 2013}

\begin{document}

\begin{abstract}
In this paper we present a combinatorial
proof of the Kronecker--Weber Theorem
for global fields of positive characteristic.
The main tools are the use of Witt vectors and their
arithmetic developed by H. L. Schmid. The key
result is to obtain, using counting arguments,
how many $p$--cyclic extensions exist
of fixed degree and bounded conductor
where only one prime ramifies are there. We then
compare this
number with the number of subextensions of
cyclotomic function fields of the same
type and verify that these two numbers
are the same.

\end{abstract}

\maketitle

\section{Introduction}\label{S1}

The classical Kronecker--Weber Theorem establishes that every
finite abelian extension of ${\ma Q}$, the field of rational numbers, is 
contained in a cyclotomic field. Equivalently, the maximal abelian
extension of ${\ma Q}$ is the union of all cyclotomic
fields. In 1974, D. Hayes \cite{Hay74},
proved the analogous result for rational congruence function fields.
Hayes constructed first cyclotomic function fields as
the analogue to classical cyclotomic fields. Indeed,
the analogy was developed in the
first place by L. Carlitz in the 1930's.
The union of all these cyclotomic function fields is not the
maximal abelian extension of the rational congruence function
field $k={\ma F}_q(T)$ since all these extensions are geometric
and the infinite prime is tamely ramified. Hayes proved that
the maximal abelian extension of $k$ is the composite of
three linearly disjoint fields: the first one is
the union of all cyclotomic fields; the second one is
the union of all constant extensions
and the third one is the union of all the subfields of the
corresponding cyclotomic function fields, where the
infinite prime is totally wildly ramified. The proof of this
theorem uses Artin--Takagi's reciprocity law
in class field theory.

In the classical case, possibly the simplest proof of the 
Kronecker--Weber Theorem uses ramification groups.
The key tool in the proof is that there is only one cyclic
extension of ${\ma Q}$ of degree $p$, $p$ an odd prime,
where $p$ is the only ramified prime. Indeed, this field
is the unique subfield of degree $p$ of the cyclotomic
field of $p^2$--roots of unity.
In the case of function fields the situation is quite different.
There exist infinitely many cyclic extensions of $k$ of degree $p$
where only one fixed prime divisor is ramified.

In this paper we present a proof of the Kronocker--Weber
Theorem analogue for rational congruence function fields using
counting arguments for the case of wild ramification.
First, similarly to the classical case, we prove that a 
finite abelian tamely ramified extension of $k$ is contained in
the composite of a cyclotomic function field and a constant
extension (see \cite{SaRzVi2012}). Next, the key part, is to show that every 
cyclic extension of $p$--power degree, where there
is only one ramified prime and it is fully ramified, and
the infinite prime is fully decomposed, is contained
in a cyclotomic function field. The particular case of an
Artin--Schreier extension was completely solved in \cite{SaRzVi2012-2}
using counting techniques.
In this paper we present another proof for the Artin--Schreier case
that also uses counting techniques but is 
suitable of generalization for cyclic extensions of degree $p^n$.
Once the latter is proven, the
rest of the proof follows easily. We use the arithmetic
of Witt vectors developed by Schmid in \cite{Sch36} to give
two proofs, one by induction and the other is direct.

\section{The result}\label{S2}

We give first some notation and some results in the theory of 
cyclotomic function fields developed by D. Hayes
\cite{Hay74}. See also \cite{Vil2006}. Let $k={\ma F}_q(T)$ be
a congruence rational function field, ${\ma F}_q$ denoting the finite
field of $q=p^s$ elements, where $p$ is the characteristic.
Let $R_T={\ma F}_q[T]$ be the ring of 
polynomials, that is, $R_T$ is
 the ring of integers of $k$. 
 For $N\in R_T\setminus
\{0\}$, $\Lambda_N$ denotes the $N$--torsion of the Carlitz
module and $k(\Lambda_N)$ denotes the $N$--th cyclotomic
function field. The $R_T$--module $\Lambda_N$ is cyclic.
For any $m\in{\ma N}$,
$C_m$ denotes a cyclic group of order $m$. Let
$K_T:=\bigcup_{M\in R_T} k(\Lambda_M)$ and ${\ma F}_{\infty}:=
\bigcup_{m\in{\ma N}}{\ma F}_{q^m}$.

We denote by ${\eu p}_{\infty}$
the pole divisor of $T$ in $k$. 
In $k(\Lambda_N)/k$, $\p$ has ramification index $q-1$
and decomposes into $\frac{|G_N|}{q-1}$ different
prime divisors of $k(\Lambda_N)$ of degree $1$, where
$G_N:=\Gal(k(\Lambda_N/k))$. Furthermore,
with the identification $G_N\cong \G N$, the inertia
group ${\eu I}$ of ${\eu p}_{\infty}$ is ${\ma F}_q^{\ast}\subseteq \G N$,
that is, ${\eu I}=\{\sigma_a\mid a\in {\ma F}_q^{\ast}\}$,
where for $A\in R_T$ we use the notation $\sigma_A(\lambda)=
\lambda^A$ for $\lambda \in \Lambda_N$.
We denote by $R_T^+$ the set of monic irreducible polynomials
in $R_T$. The
primes that ramify in $k(\Lambda_N)/k$ are ${\eu p}_{\infty}$
and the polynomials $P\in R_T^+$ such that $P\mid N$,
except in the extreme case $q=2$, $N\in\{T,T+1,T(T+1)\}$
because in this case we have $k(\Lambda_N)=k$.

We set $L_n$ to be the largest
subfield of $k\big(\Lambda_{1/T^{n+1}}\big)$ where $\p$
is fully and purely wildly ramified, $n\in{\ma N}$. 
That is, $L_n=k\big(\Lambda_{1/T^{n+1}}\big)^{{\ma F}_q^{\ast}}$.
Let $L_{\infty}:=\bigcup_{n\in{\ma N}}L_n$.

For any prime divisor ${\eu p}$ in a field $K$, $v_{\eu p}$ will denote
the valuation corresponding to ${\eu p}$.

The main goal
of this paper is to prove the following result.

\begin{theorem}[Kronecker--Weber, {\cite{Hay74}},
{\cite[Theorem12.8.31]{Vil2006}}]\label{T2.1}
The maximal abelian extension $A$ of $k$ is $A=K_T
{\ma F}_{\infty} L_{\infty}$. \qed
\end{theorem}

\s

To prove Theorem \ref{T2.1} it suffices to show
that any finite abelian extension of $k$ is contained in
$k(\Lambda_N) {\ma F}_{q^m} L_n$ for some $N\in R_T$
and $m,n\in{\ma N}$.

Let $L/k$ be a finite abelian extension. Let 
\[
G:=\Gal(L/k)\cong
C_{n_1}\times \cdots\times C_{n_l}\times C_{p^{a_1}}\times
\cdots \times C_{p^{a_h}}
\]
where $\gcd (n_i,p)=1$ for $1\leq
i\leq l$ and $a_j\in{\ma N}$ for $1\leq j\leq h$. Let $S_i\subseteq
L$ be such that $\Gal(S_i/k)\cong C_{n_i}$, $1\leq i\leq l$ and let 
$R_j\subseteq L$ be such that $\Gal(R_j/k)\cong C_{p^{a_j}}$,
$1\leq j\leq h$. To prove Theorem \ref{T2.1} it is enough
to show that each $S_i$ and each $R_j$ are contained
in $k(\Lambda_N) {\ma F}_{q^m} L_n$ for some $N\in R_T$
and $m,n\in{\ma N}$.

In short, we may assume that $L/k$ is a cyclic extension
of degree $h$ where either $\gcd(h,p)=1$ or $h=p^n$ for
some $n\in{\ma N}$.

\section{Geometric tamely ramified extensions}\label{S3}

In this section, we prove Theorem \ref{T2.1} for the particular case
of a tamely ramified extension.
Let $L/k$ be an abelian extension. Let $P\in R_T$, $d:=\deg P$.

\begin{proposition}\label{P3.1}
Let $P$ be tamely ramified in $L/k$. If $e$ denotes
the ramification index of $P$ in $L$, we have
$e\mid q^d-1$.
\end{proposition}

\proof
First we consider in general an abelian extension $L/k$. Let
$G_{-1}=D$ be the decomposition group of $P$, $G_0=I$
be the inertia group and $G_i$, $i\geq 1$ be the 
ramification groups. Let ${\eu P}$ be a prime divisor
in $L$ dividing $P$. Then if ${\mathcal O}_{{\eu P}}$
denotes the valuation ring of ${\eu P}$, we have
\[
U^{(i)}=1+{\eu P}^i\subseteq {\mathcal O}_{\eu P}^{\ast}=
{\mathcal O}_{\eu P}\setminus {\eu P},
i\geq 1,  U^{(0)}={\mathcal O}_{\eu P}^{\ast}.
\]
Let $l({\eu P}):= {\mathcal O}_{\eu P}/{\eu P}$ be the
residue field at ${\eu P}$. The following are monomorphisms:
\begin{eqnarray*}
G_i/G_{i+1}&\stackrel{\varphi_i}{\hookrightarrow} &U^{(i)}/
U^{(i+1)}\cong 
\begin{cases}
l({\eu P})^{\ast}, i=0\\
{\eu P}^i/{\eu P}^{i+1}\cong l({\eu P}), i\geq 1.
\end{cases}\\
\bar{\sigma}&\mapsto& \sigma\pi/\pi
\end{eqnarray*}
where $\pi$ denotes a prime element for ${\eu P}$.

We will prove that if $G_{-1}/G_1=D/G_1$ is abelian, then
\[
\varphi=\varphi_0\colon G_0/G_1\longrightarrow U^{(0)}/U^{(1)}\cong
\big({\mathcal O}_{\eu P}/{\eu P}\big)^{\ast}
\]
satisfies that $\im \varphi\subseteq {\mathcal O}_P/(P)\cong
R_T/(P)\cong {\ma F}_{q^d}$. In particular it will follow
$\big|G_0/G_1\big|\mid \big|{\ma F}_{q^d}^{\ast}\big|=q^d-1$.

To prove this statement, note that 
\[
\Aut(({\mathcal O}_{\eu P}/
{\eu P}) / ({\mathcal O}_P/(P)))\cong\Gal(({\mathcal O}_{\eu P}/
{\eu P})/({\mathcal O}_P/(P)))=D/I=G_{-1}/G_0
\]
 (see \cite[Corollary 5.2.12]{Vil2006}).

Let $\sigma\in G_0$ and $\varphi(\bar{\sigma})=\varphi(\sigma
\bmod G_1)=[\alpha]\in \big({\mathcal O}_{\eu P}/{\eu P}\big)^{\ast}$.
Therefore $\sigma\pi\equiv \alpha\pi\bmod {\eu P}^2$.

Let $\theta\in G_{-1}=D$ be arbitrary and let $\pi_1:=\theta^{-1}
\pi$. Then $\pi_1$ is a prime element for ${\eu P}$. Since 
$\varphi$ is independent of the prime element, it follows that
$\sigma \pi_1\equiv \alpha \pi_1\bmod {\eu P}^2$,
that is $\sigma\theta^{-1}\pi\equiv \alpha\theta^{-1}\pi\bmod {\eu P}^2$.
Since $G_{-1}/G_1$ is an abelian group, we have
\[
\sigma\pi=(\theta\sigma\theta^{-1})(\pi)\equiv \theta(\alpha)\pi\bmod{\eu P}^2.
\]
Thus $\sigma\pi\equiv \theta(\alpha)\pi\bmod{\eu P}^2$ and
$\sigma\pi\equiv \alpha\pi\bmod {\eu P}^2$. It follows that
$\theta(\alpha)\equiv \alpha\bmod {\eu P}$ for all $\theta\in G_{-1}$.

If we write $\tilde{\theta}=\theta\bmod G_0$, $\tilde{\theta}[\alpha]
=[\alpha]$, that is, $[\alpha]$ is a fixed element under the action
of the group $G_{-1}/G_0\cong \Gal(({\mathcal O}_{\eu P}/
{\eu P})/({\mathcal O}_P/(P)))$. We obtain that
$[\alpha]\in {\mathcal O}_P/(P)$.
Therefore $\im \varphi\subseteq \big({\mathcal O}_P/(P)\big)^{\ast}$
and $\big|G_0/G_1\big|\mid \big|\big({\mathcal O}_P/(P)\big)^{\ast}\big|=
q^d-1$.

Finally, since $L/k$ is abelian and $P$ is tamely ramified, $G_1=\{1\}$, it
follows that $e=|G_0|=|G_0/G_1|\mid q^d-1$. \qed

Now consider a finite abelian tamely ramified extension $L/k$
where $P_1,\ldots,P_r$ are the finite ramified primes.
Let $P\in\{P_1,\ldots,P_r\}$, $\deg P=d$. Let $e$ be the ramification
index of $P$ in $L$. Then by Proposition \ref{P3.1} we have
$e\mid q^d-1$. Now $P$ is totally ramified in $k(\Lambda_P)/k$ with
ramification index $q^d-1$. In this extension $\p$ has ramification
index equal to $q-1$.

Let $k\subseteq E\subseteq k(\Lambda_P)$ with $[E:k]=e$.
Set $\tilde{{\eu P}}$ a prime divisor in $LE$ dividing $P$.
Let ${\eu q}:= \tilde{\eu P}|_E$ and ${\eu P}:=\tilde{\eu P}|_L$.
\[
\xymatrix{
{\eu P}\ar@/^1pc/@{-}[rrrr]\ar@/_1pc/@{-}[dddd]\ar@{--}[dr]&&&&\tilde{\eu P}
\ar@/^1pc/@{-}[dddd]\ar@{--}[dl]\\
&L\ar@{--}[rr]\ar@{-}[dd]&&LE\ar@{--}[dd]\\
&&M\ar@{-}[dl]\ar@{-}[ur]_H\\
&k\ar@{-}[rr]_e&&E\ar@{-}[rr]\ar@{--}[dr]&&k(\Lambda_P)\\
P\ar@/_1pc/@{-}[rrrr]\ar@{--}[ru]&&&&{\eu q}
}
\]

We have $e=e_{L/k}({\eu P}|P)=e_{E/k}({\eu q}|P)$. By
Abhyankar's Lemma \cite[Theorem 12.4.4]{Vil2006}, we obtain
\[
e_{LE/k}(\tilde{\eu P}|P)=\lcm [e_{L/k}({\eu P}|P), e_{E/k}(
{\eu q}|P)]=\lcm[e,e]=e.
\]
Let $H\subseteq \Gal(LE/k)$ be the inertia group of $\tilde{
\eu P}/P$. Set $M:=(LE)^H$. Then $P$ is unramified in $M/k$.
We want to see that $L\subseteq Mk(\Lambda_P)$. Indeed we have
$[LE:M]=e$ and $E\cap M=k$ since $P$ is totally ramified
in $E/k$ and unramified in $M/k$. It follows that
$[ME:k]=[M:k][E:k]$. Therefore
\[
[LE:k]=[LE:M][M:k]=e\frac{[ME:k]}{[E:k]}=
e\frac{[ME:k]}{e}=[ME:k].
\]
Since $ME\subseteq LE$ it follows that
$LE=ME=EM\subseteq k(\Lambda_P)M$. Thus
$L\subseteq k(\Lambda_P)M$.

In $M/k$ the finite ramified primes are
$\{P_2,\cdots, P_r\}$. In case $r-1\geq 1$,
we may apply the above argument to $M/k$ and we 
obtain $M_2/k$ such that at most $r-2$ finite
primes are ramified and $M\subseteq k(\Lambda_{P_2})M_2$,
so that $L\subseteq k(\Lambda_{P_1})M\subseteq
k(\Lambda_{P_1})k(\Lambda_{P_2})M_2=
k(\Lambda_{P_1P_2})M_2$. 

Performing the above process at most $r$ times we 
have
\begin{equation}\label{E3.2}
L\subseteq k(\Lambda_{P_1P_2\cdots P_r})M_0
\end{equation}
where in $M_0/k$ the only possible ramified prime is $\p$.

We also have

\begin{proposition}\label{P3.3}
Let $L/k$ be an abelian extension where at most
one prime divisor ${\eu p}_0$ of degree $1$ is
ramified and the extension is tamely ramified. Then
$L/k$ is a constant extension.
\end{proposition}

\begin{proof}
By Proposition \ref{P3.1} we have $e:=e_{L/k}({\eu p}_0)|q-1$.
Let $H$ be the inertia group of ${\eu p}_0$. Then $|H|=e$ and
${\eu p}_0$ is unramified in $E:=L^H/k$. Therefore $E/k$
is an unramified extension. Thus $E/k$ is a constant
extension.

Let $[E:k]=m$. Then if ${\eu P}_0$ is a prime divisor in $E$
dividing ${\eu p}_0$ then the relative degree $d_{E/k}
({\eu P}_0|{\eu p}_0)$ is equal to $m$, the number of
prime divisors in $E/k$ is $1$ and the degree of ${\eu P}_0$
is $1$ (see \cite[Theorem 6.2.1]{Vil2006}).
Therefore ${\eu P}_0$ is the only prime divisor 
ramified in $L/E$ and it is of degree $1$ and totally ramified.
Furthermore $[L:E]=e\mid q-1=|{\ma F}_q^{\ast}|$.

The $(q-1)$-th roots of unity belong to $\F\subseteq k$.
Hence $k$ contains the $e$--th roots of unity and $L/E$
is a Kummer extension, say $L=E(y)$ with $y^e=\alpha
\in E=k{\ma F}_{q^m}={\ma F}_{q^m}(T)$.
We write $\alpha$ in a normal form as prescribed by Hasse
\cite{Has35}: $(\alpha)_E=\frac{{\eu P}_0^a
{\eu a}}{{\eu b}}$, $0<a<e$. Now since $\deg (\alpha)_E=0$
it follows that $\deg_E {\eu a}$ or $\deg_E {\eu b}$ is not
a multiple of $e$. This contradicts that ${\eu p}_0$ is the
only ramified prime. Therefore $L/k$ is a constant extension.
\end{proof}

As a corollary to (\ref{E3.2}) and Proposition \ref{P3.3}
we obtain

\begin{corollary}\label{C3.5}
If $L/k$ is a finite abelian tamely ramified extension where
the ramified finite prime divisors are $P_1,\ldots,P_r$, then
\[
L\subseteq k(\Lambda_{P_1\cdots P_r}) {\ma F}_{q^m},
\]
for some $m\in{\ma N}$.\hfill\qed
\end{corollary}

\section{Reduction steps}\label{S2(1)}

As a consequence of Corollary \ref{C3.5}, Theorem
\ref{T2.1} will follow if we prove it for the particular
case of a cyclic extension $L/k$ of degree $p^n$ for
some $n\in{\ma N}$. Now, this kind of extensions are
given by a Witt vector:
\[
K=k(\vec y)=k(y_1,\ldots,y_n) \quad\text{with}\quad
 \vec y^p\Witt - \vec y =
\vec \beta=(\beta_1,\ldots,\beta_n)\in W_n(k)
\]
 where
for any field $E$ of characteristic $p$,
$W_n(E)$ denotes the ring of Witt vectors
of length $n$ with components in $E$.

The following result was proved in \cite{MaRzVi2013}. It
``separates'' the ramified prime divisors.

\begin{theorem}\label{T2.3} Let $K/k$ be a cyclic extension of degree $p^n$
where $P_1,\ldots,P_r\in R_T^+$ and possibly $\p$, are the ramified prime
divisors. Then $K=k(\vec y)$ where
\[
\vec y^p\Witt -\vec y=\vec \beta={\vec\delta}_1\Witt + \cdots \Witt + {\vec\delta}_r
\Witt + \vec\mu,
\]
with $\delta_{ij}=\frac{Q_{ij}}{P_i^{e_{ij}}}$, $e_{ij}\geq 0$, $Q_{ij}\in R_T$
and
\l
\item  if $e_{ij}=0$ then $Q_{ij}=0$;
\item if $e_{ij}>0$ then $p\nmid e_{ij}$, $\gcd(Q_{ij},P_i)=1$ and 
$\deg (Q_{ij})<\deg (P_i^{e_{ij}})$, 
\end{list}
and ${\mu}_j=f_j(T)\in R_T$ with
\l
\setcounter{bean}{2}
\item $p\nmid \deg f_j$ when $f_j\not\in {\ma F}_q$ and
\item $\mu_j\notin\wp({\ma F}_q):=\{a^p-a\mid a\in{\ma F}_q\}$ when 
$\mu_j\in {\ma F}_q^{\ast}$.\qed
\end{list}
\end{theorem}

\s

Consider the field $K=k(\vec y)$ as above, where precisely
one prime divisor $P\in R_T^+$ ramifies, with
\begin{gather}\label{EqNew1}
\begin{array}{l}
\text{$\beta_i=
\frac{Q_i}{P^{\lambda_i}}$, $Q_i\in R_T$
such that $\lambda_i\geq 0$},\\
\text{if $\lambda_i=0$ then $Q_i=0$,}\\
\text{if $\lambda_i>0$ then $\gcd(\lambda_i,p)=1$,
$\gcd(Q_i,P)=1$ and $\deg Q_i<\deg P^{\lambda_i}$},\\
\lambda_1>0.
\end{array}
\end{gather}

A particular case of Theorem \ref{T2.3} suitable for our study
is given in the following proposition.

\begin{proposition}\label{P2.6} Assume that every extension $K_1/k$
that meets the conditions of {\rm (\ref{EqNew1})} 
satisfies that $K_1\subseteq k(\Lambda_{P^{\alpha}})$ for some
$\alpha\in{\ma N}$. Let $K/k$ be the extension defined by
$K=k(\vec{y})$ where $\wp(\vec{y})=  \vec y^p\Witt - \vec y=
\vec \beta$ with $\vec \beta=(\beta_1,\ldots,\beta_n)$,
$\beta_i$ given in
normal form: $\beta_i\in{\ma F}_q$
or $\beta_i=\frac{Q_i}{P^{\lambda_i}}$, 
$Q_i\in R_T$ and $\lambda_i>0$,
$\gcd(\lambda_i,p)=1$, $\gcd(Q_i,P)=1$ and $\deg Q_i\leq
\deg P^{\lambda_i}$. Then $K\subseteq {\ma F}_{q^{p^n}}
\lam \alpha$ for some $\alpha\in{\ma F}_q$.
\end{proposition}

\proof From Theorem \ref{T2.3} we have that we can
decompose the vector $\vec \beta$ as
$\vec \beta=\vec \varepsilon\Witt + \vec \gamma$ with
$\varepsilon_i\in{\ma F}_q$ for all $1\leq i\leq n$ and
$\gamma_i=0$
or $\gamma_i=\frac{Q_i}{P^{\lambda_i}}$,
$Q_i\in R_T$ and $\lambda_i>0$,
$\gcd(\lambda_i,p)=1$, $\gcd(Q_i,P)=1$ and $\deg Q_i<
\deg P^{\lambda_i}$. 

Let
$\gamma_1=\cdots=\gamma_r=0$, and
$\gamma_{r+1}\notin {\ma F}_q$.
We have $K\subseteq k(\vec \varepsilon)k(\vec \gamma)$.
Now $k(\vec \varepsilon)\subseteq {\ma F}_{q^{p^n}}$ and
$k(\vec \gamma)=k(0,\ldots,0,\gamma_{r+1},\ldots,\gamma_n)$.

For any Witt vector $\vec x=(x_1,\ldots,x_n)$ we have
the decomposition given by Witt himself
\begin{align*}
\vec x=&(x_1,0,0,\ldots,0)\Witt +(0,x_2,0,\ldots,0)\Witt +\cdots
\Witt + (0,\ldots, 0,x_j,0,\ldots,0)\\
&\Witt +(0,\ldots, 0, x_{j+1}, \ldots, x_n)
\end{align*}
for each $0\leq j\leq n-1$. It follows that
$k(\vec \gamma)=k(\gamma_{r+1},\ldots,\gamma_n)$.
Since this field fulfills the conditions of (\ref{EqNew1}),
we have $k(\vec \gamma)\subseteq k(\Lambda_{P^{\alpha}})$
for some $\alpha\in{\ma N}$. The result follows. \qed

\s

\begin{remark}\label{R2.6'}
The prime $\p$ can be handled in the same way. The conditions
(\ref{EqNew1}) for $\p$ are the following. Let $K=k(\vec \mu)$
with $\mu_j=f_j(T)\in R_T$, with $f_j(0)=0$ for all $j$ and
either $f_j(T)=0$ or $f_j(T)\neq 0$ and $p\nmid \deg f_j(T)$.
The condition $f_j(0)=0$ means that the infinite prime for
$T^{\prime}=1/T$ is either decomposed or ramified in each layer,
that is, its inertia degree is $1$ in $K/k$.
In this case with the change of variable $T^{\prime}=1/T$
the hypotheses in Proposition \ref{P2.6} say that
any field meeting these conditions satisfies that
$K\subseteq k(\Lambda_{T^{\prime m}})=
k(\Lambda_{T^{-m}})$ for some $m\in{\ma N}$.
However, since the degree of the extension $K/k$
is a power of $p$ we must have that $K\subseteq k(\Lambda_{T^{-m}})^{
{\ma F}_q^{\ast}}=L_{m-1}$.
\end{remark}

With the notation of Theorem \ref{T2.3}
 we obtain that if $\vec z_i^p\Witt - \vec z_i
=\vec \delta_i$, $1\leq i\leq r$ and if $\vec v^p\Witt -\vec v=\vec \mu$,
then $L=k(\vec y)\subseteq k(\vec z_1,\ldots,\vec z_r,\vec v)=
k(\vec z_1)\ldots k(\vec z_r) k(\vec v)$. Therefore if Theorem
\ref{T2.1} holds for each $k(\vec z_i)$, $1\leq i\leq r$ and for
$k(\vec v)$, then it holds for $L$.

From Theorem \ref{T2.3}, Proposition \ref{P2.6} and
the remark after this proposition, we obtain that
to prove Theorem \ref{T2.1} it suffices to show that
any field extension $K/k$ meeting the
conditions of (\ref{EqNew1}) satisfies that 
either $K\subseteq k(\Lambda_{P^{\alpha}})$
for some $\alpha\in{\ma N}$ or $K\subseteq L_m$
for some $m\in{\ma N}$.
It suffices to study the case for $P\in R_T$.

Next we study the behavior of $\p$ in
an arbitrary cyclic
extension $K/k$ of degree $p^n$.

Consider first the case $K/k$ cyclic of degree $p$.
Then $K=k(y)$ where $y^p-y=\alpha\in k$. 
The equation can be 
normalized as:
\begin{equation}\label{E2.2}
y^p-y=\alpha=\sum_{i=1}^r\frac{Q_i}{P_i^{e_i}} + f(T)=
\frac{Q}{P_1^{e_1}\cdots P_r^{e_r}}+f(T),
\end{equation}
where $P_i\in R_T^+$, $Q_i\in R_T$, 
$\gcd(P_i,Q_i)=1$, $e_i>0$, $p\nmid e_i$, $\deg Q_i<
\deg P_i^{e_i}$, $1\leq i\leq r$,
$\deg Q<\sum_{i=1}^r\deg P_i^{e_i}$, $f(T)\in R_T$,
with $p\nmid \deg f$ when $f(T)\not\in {\ma F}_q$
and $f(T)\notin \wp({\ma F}_q)$ when $f(T)
\in {\ma F}_q^{\ast}$.

We have that the finite primes ramified in $K/k$ are precisely
$P_1,\ldots,P_r$ (see \cite{Has35}). With respect to $\p$ we have
the following well known result. We present a proof for the
sake of completeness.

\begin{proposition}\label{P2.4}
Let $K=k(y)$ be given by {\rm (\ref{E2.2})}. Then
the prime $\p$ is
\l
\item decomposed if $f(T)=0$.
\item inert if $f(T)\in {\ma F}_q$ and $f(T)\not\in \wp({\ma F}_q)$.
\item ramified if $f(T)\not\in {\ma F}_q$ (thus $p\nmid\deg f$).
\end{list}
\end{proposition}

\proof
First consider the case $f(T)=0$. Then $v_{\p}(\alpha)=
\deg(P_1^{e_1}\cdots P_r^{e_r})-\deg Q>0$.
Therefore $\p$ is unramified. Now $y^p-y=\prod_{i=0}^{p-1}
(y-i)$. Let ${\eu P}_{\infty}\mid\p$. Then 
\[
v_{{\eu P}_{\infty}}(y^p-y)=\sum_{i=0}^{p-1}v_{{\eu P}_{\infty}}
(y-i)=e({\eu P}_{\infty}|\p)v_{\p}(\alpha)=v_{\p}(\alpha)>0.
\]
Therefore, there exists $0\leq i\leq p-1$ such that
$v_{\p}(y-i)>0$. Without loss of generality we may assume that $i=0$.
Let $\sigma\in \Gal(K/k)\setminus \{\Id\}$. Assume that
${\eu P}_{\infty}^{\sigma}={\eu P}_{\infty}$. We have $y^{\sigma}=
y-j$, $j\neq 0$. Thus, on the one hand
\[
v_{{\eu P}_{\infty}}(y-j)=v_{{\eu P}_{\infty}}(y^{\sigma})=
v_{\sigma({\eu P}_{\infty})}(y)=v_{{\eu P}_{\infty}}(y)>0.
\]
On the other hand, since
$v_{{\eu P}_{\infty}}(y)>0=v_{{\eu P}_{\infty}}(j)$,
it follows that 
\[
v_{{\eu P}_{\infty}}(y-j)=\min\{
v_{{\eu P}_{\infty}}(y),v_{{\eu P}_{\infty}}(j)\}=0.
\]
This contradiction shows that
${\eu P}_{\infty}^{\sigma}\neq {\eu P}_{\infty}$
so that $\p$ decomposes in $K/k$.

Now we consider the case $f(T)\neq 0$. If $f(T)\not\in \F$, then
$\p$ ramifies since it is in the normal form prescribed by
Hasse \cite{Has35}.

The last case is when $f(T)\in\F$, $f(T)\not\in \wp(\F)$. Let
$b\in {\ma F}_{q^p}$ with $b^p-b=a=f(T)$. Since $\deg \p=1$,
$\p$ is inert in the constant extension $k(b)/k$ (\cite[Theorem 6.2.1]{Vil2006}).
Assume that $\p$ decomposes in $k(y)/k$. We have 
the following diagram
\[
\xymatrix{
k(y)\ar@{-}[r]^{\p}_{\text{inert}}\ar@{-}[d]_{\substack{\p\\ \text{decomposes}}}
&k(y,b)\ar@{-}[d]\\ k\ar@{-}[r]^{\p}_{\text{inert}}&k(b)
}
\]

The decomposition group of $\p$ in $k(y,b)/k$ is
$\Gal(k(y,b)/k(y))$. Therefore $\p$ is inert  in every
field of degree $p$ over $k$ other than $k(y)$. Since the
fields of degree $p$ are $k(y+ib), k(b)$, $0\leq i\leq p-1$,
in $k(y+b)/k$ we have
\[
(y+b)^p-(y+b)=(y^p-y)+(b^p-b)=\alpha-a =\frac{Q}{P_1^{
e_1}\cdots P_r^{e_r}}
\]
with $\deg (\alpha-a)<0$. Hence, by the first part, $\p$ decomposes
in $k(y+b)/k$ and in $k(y)/k$ which is impossible.
Thus $\p$ is inert in $k(y)/k$. \qed

\s

The general case for the behavior of $\p$ 
in a cyclic $p$--extension is given in
\cite[Proposition 5.6]{MaRzVi2013} and it is a consequence of Proposition
\ref{P2.4}.

\begin{proposition}\label{P2.5}
Let $K/k$ be given as in Theorem {\rm{\ref{T2.3}}}. Let 
$\mu_1=\cdots=\mu_s=0$, $\mu_{s+1}\in {\ma F}_q^{\ast}$, 
$\mu_{s+1}\not\in \wp({\ma F}_q)$ and finally, let $t+1$ be the
first index with $f_{t+1}\not\in{\ma F}_q$ (and therefore $p\nmid \deg f_{t+1}$).
Then the ramification index of $\p$ is $p^{n-t}$, the inertia degree
of $\p$ is $p^{t-s}$ and the decomposition number
of $\p$ is $p^s$. More precisely, if $\Gal(K/k)=\langle \sigma\rangle
\cong C_{p^n}$, then the inertia group of $\p$ is ${\eu I}=\langle \sigma^{p^t}
\rangle$ and the decomposition group of $\p$ is ${\eu D}=\langle
\sigma^{p^s}\rangle$. \qed
\end{proposition}

\s

From Propositions \ref{P2.4} and \ref{P2.5} we obtain

\begin{proposition}\label{P2.5'}
If $K$ is a field defined by an equation of the
type given in {\rm (\ref{EqNew1})}, then 
$K/k$ is a cyclic extension of degree $p^n$, $P$
is the only ramified prime, it is fully ramified and $\p$ is 
fully decomposed. 

Similarly, if $K=k(\vec v)$ where $v_i=f_i(T)\in R_T$, $f_i(0)=0$
for all $1\leq i\leq n$ and $f_1(T)\notin {\ma F}_q$,
$p\nmid \deg f_1(T)$, then $\p$ is the only
ramified prime in $K/k$, it is fully ramified and the
zero divisor of $T$ which is the infinite prime in $R_{1/T}$,
is fully decomposed.\qed
\end{proposition}

\s

We have reduced the proof of Theorem \ref{T2.1} to prove 
that any extension of the type given in Proposition \ref{P2.5'} is 
contained in either $k(\Lambda_{P^{\alpha}})$ for
some $\alpha\in{\ma N}$ or in $L_m$ for some 
$m\in{\ma N}$. The second case is a consequence of the first
one with the change of variable $T^{\prime}=1/T$.

Let $n,\alpha\in{\ma N}$. Denote by $v_n(\alpha)$ 
the number of cyclic groups of order $p^n$ 
contained in $\G {P^{\alpha}}\cong
\Gal(k(\Lambda_{P^{\alpha}}/k)$. We have
that $v_n(\alpha)$ is the number of cyclic field
extensions $K/k$ of degree $p^n$
and $K\subseteq k(\Lambda_{P^{\alpha}})$.
Every such extension satisfies that its 
conductor ${\eu F}_K$ divides $P^{\alpha}$.

Let $t_n(\alpha)$ be the number of cyclic field
extensions $K/k$ of degree $p^n$ such that $P$ is the only
ramified prime, it is fully ramified, $\p$ is
fully decomposed and its conductor
${\eu F}_K$ is a divisor of $P^{\alpha}$.
Since every cyclic extension $K/k$ of degree
$p^n$ such that $k\subseteq K\subseteq
k(\Lambda_{P^{\alpha}})$ satisfies these
conditions we have $v_n(\alpha)\leq t_n(\alpha)$.
If we prove $t_n(\alpha)\leq v_n(\alpha)$ then
every extension satisfying equation (\ref{EqNew1})
is contained in a cyclotomic extension and
Theorem \ref{T2.1} follows.

Therefore, to prove Theorem \ref{T2.1}, it suffices to
prove
\begin{gather}\label{EqNew2}
t_n(\alpha)\leq v_n(\alpha)\quad\text{for all}\quad
n, \alpha\in{\ma N}.
\end{gather}

\section{Wildly ramified extensions}\label{S4}

In this section we prove (\ref{EqNew2}) by induction on $n$ and as 
a consequence we obtain 
our main result, Theorem \ref{T2.1}.
First we compute $v_n(\alpha)$ for all $n,\alpha\in{\ma N}$.

\begin{proposition}\label{P4.1}
The number $v_n(\alpha)$ of different cyclic groups of order $p^n$ 
contained in $\G {P^{\alpha}}$ is
\[
v_n(\alpha)=\frac{q^{d\big(\alpha-\mayorchico{\alpha}
{p^n}\big)}-q^{d\big(\alpha-\mayorchico{\alpha}
{p^{n-1}}\big)}}{p^{n-1}(p-1\big)}=
\frac{q^{d\big(\alpha-\mayorchico{\alpha}
{p^{n-1}}\big)}\big(q^{d\big(\mayorchico{\alpha}{p^{n-1}}
-\mayorchico{\alpha}{p^{n}}\big)}-1\big)}{p^{n-1}(p-1)},
\]
where $\lceil x\rceil$ denotes the {\em ceiling function},
that is, $\lceil x\rceil$ denotes the minimum integer greater
than or equal to $x$.
\end{proposition}

\proof
Let $P\in R_T^+$ and $\alpha\in{\ma N}$ with $\deg P=d$.
First we consider how many cyclic extensions of degree $p^n$
are contained in $\lam {\alpha}$. Since $\p$ is tamely
ramified in $\lam \alpha$, if $K/k$ is a cyclic extension
of degree $p^n$, $\p$ decomposes fully in $K/k$
(\cite[Theorem 12.4.6]{Vil2006}). We have
$\Gal(\lam \alpha/k)\cong \G{P^{\alpha}}$ and
the exact sequence
\begin{gather}\label{E4.1}
0\longrightarrow D_{P,P^{\alpha}}\longrightarrow
\G {P^{\alpha}}\stackrel{\varphi}{\longrightarrow} \G P\longrightarrow 0,
\end{gather}
where
\begin{eqnarray*}
\varphi\colon \G {P^{\alpha}}&\longrightarrow&\G P\\
A\bmod P^{\alpha}&\longmapsto& A\bmod P
\end{eqnarray*}
and $D_{P,P^{\alpha}}=\{N\bmod P^{\alpha}\mid
N\equiv 1\bmod P\}$. We safely may consider
$D_{P,P^{\alpha}}=\{1+hP\mid h\in R_T, \deg h<\deg P^{\alpha}=
d\alpha\}$.

We have $\G {P^{\alpha}}\cong \G P \times D_{P,P^{\alpha}}$
and $\G P\cong C_{q^d-1}$.
First we compute how many elements of order $p^n$
contains $\G {P^{\alpha}}$. These elements belong to 
$D_{P,P^{\alpha}}$. Let $A=1+hP\in D_{P,P^{\alpha}}$
of order $p^n$. We write $h=gP^{\gamma}$ with
$g\in R_T$, $\gcd (g,P)=1$ and $\gamma\geq 0$.
We have $A=1+gP^{1+\gamma}$. Since $A$ is
of order $p^n$, it follows that
\begin{gather}
A^{p^n}=1+g^{p^n}P^{p^n(1+\gamma)}\equiv 1\bmod P^{\alpha}\label{E4.2}\\
\intertext{and}
A^{p^{n-1}}=1+g^{p^{n-1}}P^{p^{n-1}(1+\gamma)}
\not\equiv 1\bmod P^{\alpha}\label{E4.3}.
\end{gather}

From (\ref{E4.2}) and (\ref{E4.3}) it follows that
\begin{gather}\label{E4.4}
p^{n-1}(1+\gamma)<\alpha\leq p^n(1+\gamma),\\
\intertext{which is equivalent to}
\mayor{\alpha}{p^n}-1\leq \gamma <\mayor{\alpha}
{p^{n-1}}-1.
\end{gather}
 Observe that for the existence
of at least one element of order $p^n$ we need
$\alpha>p^{n-1}$.

Now, for each $\gamma$ satisfying (\ref{E4.4}) we have
$\gcd (g,P)=1$ and $\deg g + d(1+\gamma) <d \alpha$,
that is, $\deg g<d(\alpha-\gamma-1)$. Thus, there exist
$\Phi(P^{\alpha-\gamma-1})$ such $g$'s, where
for any $N\in R_T$, $\Phi(N)$ denotes
the order of $\big|\G N\big|$.

Therefore the number of elements of order $p^n$ 
in $D_{P,P^{\alpha}}$ is 
\begin{gather}\label{E4.5}
\sum_{\gamma=\mayorchico{\alpha}{p^n}-1}^{
\mayorchico{\alpha}{p^{n-1}}-2}\Phi(P^{\alpha-\gamma-1})=
\sum_{\gamma'=\alpha-\mayorchico{\alpha}{p^{n-1}}+1}^{
\alpha-\mayorchico{\alpha}{p^n}}\Phi(P^{\gamma'}).
\end{gather}

Note that for any $1\leq r\leq s$ we have
\begin{align*}
\sum_{i=r}^{s}\Phi(P^i)&=\sum_{i=r}^s q^{d(i-1)}(q^d-1)=
(q^d-1)q^{d(r-1)}\sum_{j=0}^{s-r}q^{dj}\\
&=(q^d-1)q^{d(r-1)}\frac{q^{d(s-r+1)}-1}{q^d-1}=q^{ds}-q^{d(r-1)}.
\end{align*}
Hence (\ref{E4.5}) is equal to
\[
q^{d\big(\alpha-\mayorchico{\alpha}{p^n}\big)}-q^{d\big(\alpha-\mayorchico{\alpha}
{p^{n-1}}\big)}=q^{d\big(\alpha-\mayorchico{\alpha}
{p^{n-1}}\big)}(q^{d\big(\mayorchico{\alpha}{p^{n-1}}
-\mayorchico{\alpha}{p^{n}}\big)}-1).
\]

Since each cyclic group of order $p^n$ has $\varphi(p^n)=
p^{n-1}(p-1)$ generators, we obtain the result. \qed

\s

Note that if $K/k$ is any field contained in $\lam {\alpha}$ then the
conductor ${\eu F}_K$ of $K$ is a divisor of $P^{\alpha}$.

During the proof of Proposition \ref{P4.3}
we compute $t_1(\alpha)$, that is,
the number of cyclic extensions $K/k$ of degree $p$
such that $P$ is the only ramified prime (it is fully ramified),
$\p$ is decomposed in
$K/k$ and ${\eu F}_K\mid P^{\alpha}$ and we obtain
(\ref{EqNew2}) for the case $n=1$.
We have already solved this case in \cite{SaRzVi2012-2}. Here
we present another proof, which is suitable of 
generalization to the case of cyclic extensions of
degree $p^n$.

\begin{proposition}\label{P4.3}
Every cyclic extension $K/k$ of degree $p$ such that $P$ is the only
ramified prime, $\p$ decomposes in $K/k$ and ${\eu F}_K\mid P^{\alpha}$
is contained in $\lam \alpha$.
\end{proposition}

\proof
From the Artin--Schreier theory
(see (\ref{E2.2})) and Proposition \ref{P2.4},
we have that the field $K$ is given by
$K=k(y)$ with the Artin--Schreier equation of 
$y$ normalized as prescribed by Hasse \cite{Has35}. That is
\[
y^p-y=\frac{Q}{P^{\lambda}},
\]
where $P\in R_T^+$, $Q\in R_T$, 
$\gcd(P,Q)=1$, $\lambda >0$, $p\nmid \lambda$, $\deg Q<
\deg P^{\lambda}$. Now the conductor ${\eu F}_K$ satisfies
 ${\eu F}_K=P^{\lambda+1}$ so
$\lambda\leq \alpha-1$.

We have that if $K=k(z)$ with $z^p-z=a$ then there exist
$j\in {\ma F}_p^{\ast}$ and $c\in k$ such that 
$z=jy+c$ and $a=j\frac{Q}{P^{\lambda}}+\wp(c)$ where
$\wp (c)=c^p-c$. If $a$ is also given in normal form
then $c=\frac{h}{P^{\gamma}}$ with $p\gamma\leq \lambda$
(indeed, $p\gamma<\lambda$ since $\gcd (\lambda,p)=1$)
and $\deg h<\deg P^{\gamma}$ or $h=0$. Let $\gamma_0:=
\integerchico{\lambda}{p}$. Then any such $c$ can be
written as $c=\frac{hP^{\gamma_0-\gamma}}{P^{\gamma_0}}$.
Therefore $c\in {\mathcal G}:=\Big\{\frac{h}{P^{\gamma_0}}\mid
h\in R_T, \deg h<\deg P^{\gamma_0}=d\gamma_0
\text{\ or\ } h=0\Big\}$.

If $c\in{\mathcal G}$ and $j\in\{1,2,\ldots,p-1\}$ we have
\begin{align*}
a&=j\frac{Q}{P^{\lambda}}+\wp(c)=j\frac{Q}{P^{\lambda}}+
\frac{h^p}{P^{p\gamma_0}}+\frac{h}{P^{\gamma_0}}\\
&=\frac{jQ+P^{\lambda-p\gamma_0}h+P^{\lambda-\gamma_0}}
{P^{\lambda}}=\frac{Q_1}{P^{\lambda}},
\end{align*}
with $\deg Q_1<\deg P^{\lambda}$.
Since $\lambda-p\gamma_0>0$ and $\lambda-\gamma_0>0$,
we have $\gcd(Q_1,P)=1$. Therefore $a$ is in normal form.

It follows that the same field has $|{\ma F}_p^{\ast}||\wp({\mathcal
G})|$ different representations in standard form.
Now ${\mathcal G}$ and $\wp({\mathcal G})$ are additive
groups and $\wp\colon {\mathcal G}\to \wp({\mathcal G})$ is a 
group epimorphism with kernel $\ker \wp = {\mathcal G}\cap
\{c\mid \wp(c)=c^p-c=0\}={\mathcal G}\cap {\ma F}_p=\{0\}$.
We have $|\wp({\mathcal G})|=|{\mathcal G}|=|R_T/(P^{\gamma_0})|=
q^{d\gamma_0}$.

From the above discussion we obtain that the number of different
cyclic extensions $K/k$ of degree $p$ such that the conductor
of $K$ is ${\eu F}_K= P^{\lambda+1}$ is equal to
\begin{equation}\label{E4.8}
\frac{\Phi(P^{\lambda})}{|{\ma F}_p^{\ast}||\wp({\mathcal G})|}=
\frac{q^{d(\lambda-1)}(q^d-1)}{(p-1)q^{d\gamma_0}}=
\frac{q^{d(\lambda-\integerchico{\lambda}{p}-1)}(q^d-1)}{p-1}=
\frac{1}{p-1}\Phi\big(P^{\lambda-\integerchico{\lambda}{p}}\big).
\end{equation}

Therefore the number of different
cyclic extensions $K/k$ of degree $p$ such that the conductor
${\eu F}_K$ of $K$ is a divisor of $P^{\alpha}$ is given by
$t_1(\alpha)=\frac{w(\alpha)}{p-1}$ where
\begin{equation}\label{E3.10}
w(\alpha)=\sum_{\substack{\lambda=1\\ \gcd(\lambda,p)=1}}^{\alpha-1}
\Phi\big(P^{\lambda-\integerchico{\lambda}{p}}\big).
\end{equation}

To compute $w(\alpha)$ write $\alpha-1=pt_0+r_0$ with $t_0\geq 0$
and $0\leq r_0\leq p-1$. Now $\{\lambda\mid 1\leq \lambda \leq 
\alpha-1, \gcd(\lambda,p)=1\}={\mathcal A}\cup {\mathcal B}$ where
\begin{gather*}
{\mathcal A}
=\{pt+r\mid 0\leq t\leq t_0-1, 1\leq r\leq p-1\}
\quad \text{and}\quad {\mathcal B}=
\{pt_0+r\mid 1\leq r\leq r_0\}.
\intertext{Then}
w(\alpha)=\sum_{\lambda\in{\mathcal A}}
\Phi\big(P^{\lambda-\integerchico{\lambda}{p}}\big)+\sum_{\lambda\in
{\mathcal B}}\Phi\big(P^{\lambda-\integerchico{\lambda}{p}}\big)
\end{gather*}
where we understand that if a set, ${\mathcal A}$ or
${\mathcal B}$ is empty, the respective sum is $0$.

Then 
\begin{align}\label{E4.8'}
w(\alpha)&=\sum_{\substack{0\leq t\leq t_0-1\\ 1\leq r\leq p-1}}
q^{d(pt+r-t-1)}(q^d-1)+\sum_{r=1}^{r_0}(q^{d(pt_0+r-t_0-1})(q^d-1)\nonumber\\
&=(q^d-1)\Big(\sum_{t=0}^{t_0-1}q^{d(p-1)t}\Big)\Big(\sum_{
r=1}^{p-1}q^{d(r-1)}\Big)+
(q^d-1)q^{d(p-1)t_0}\sum_{r=1}^{r_0}q^{d(r-1)}\\
&=(q^d-1)\frac{q^{d(p-1)t_0}-1}{q^{d(p-1)}-1}\frac{q^{d(p-1)}-1}{q^d-1}
+(q^d-1)q^{d(p-1)t_0}\frac{q^{dr_0}-1}{q^d-1}\nonumber\\
&=q^{d((p-1)t_0+r_0)}-1=q^{d(pt_0+r_0-t_0)}-1=
q^{d\big(\alpha-1-\integerchico
{\alpha-1}{p}\big)}-1.\nonumber
\end{align}

Therefore, the number of different cyclic extensions $K/k$ of degree $p$
such that $P$ is the only ramified prime, ${\eu F}_K\mid P^{\alpha}$ and
$\p$ decomposes, is
\begin{equation}\label{E4.7}
t_1(\alpha)=\frac{w(\alpha)}{p-1}=\frac{q^{d(\alpha-1-\integerchico
{\alpha-1}{p})}-1}{p-1}.
\end{equation}

To finish the proof of Proposition \ref{P4.3} we need the following

\begin{lemma}\label{L4.2}
For any $\alpha\in{\ma Z}$ and $s\in {\ma N}$ we have
\l
\item $\integer{\integerchico{\alpha}{p^s}}{p}=
\integer{\integerchico{\alpha}{p}}{p^s}=\integer{\alpha}{p^{s+1}}$.

\item $\mayor{\alpha}{p^s}=\integer{\alpha-1}{p^s}+1$.
\end{list}
\end{lemma}

\proof
For (a), we prove only $\integer{\integerchico{\alpha}{p^s}}{p}
=\integer{\alpha}{p^{s+1}}$, the other equality is similar.
Note that the case $s=0$ is clear. Set $\alpha=tp^{s+1}+r$
with $0\leq r\leq p^{s+1}-1$. Let $r=lp^s+r'$ with $0\leq r'\leq p^s-1$. Note
that $0\leq l\leq p-1$. Hence $\alpha=tp^{s+1}+lp^s+r'$, $0\leq r'\leq p^s-1$
and $0\leq l\leq p-1$. Therefore $\integer{\alpha}{p^s}=tp+l$, and
$\d\frac{\integer{\alpha}{p^s}}{p}=t+\frac{l}{p}$, $0\leq l\leq p-1$. Therefore
$\integer{\integerchico{\alpha}{p^s}}{p}=t=\integer{\alpha}{p^{s+1}}$.

For (b) write $\alpha=p^s t+r$ with
$0\leq r\leq p^s-1$. If $p^s\mid \alpha$ then $r=0$ and $\mayor{\alpha}{p^s}
=t$, $\integer{\alpha-1}{p^s}=\integer{p^s t-1}{p^s}=\d\Big[t-\frac{1}{p^s}\Big]=
t-1=\mayor{\alpha}{p^s}-1$.

If $ p^s\nmid \alpha$, then $1\leq r\leq p^s-1$ and $\alpha-1=p^st+(r-1)$
with $0\leq r-1\leq p^s-2$. Thus $\mayor{\alpha}{p^s}=\d\Big\lceil t+\frac{r}{p^s}
\Big\rceil=t+1$ and $\integer{\alpha-1}{p^s}=\d\Big[t+\frac{r-1}{p^s}\Big]=t=
\mayor{\alpha}{p^s}-1$. This finishes the proof of Lemma \ref{L4.2}. \qed

\s

From Lemma \ref{L4.2} (b) we obtain that (\ref{E4.7}) is equal to
\begin{equation}\label{E3.11}
t_1(\alpha)=\frac{w(\alpha)}{p-1}=\frac{q^{d\big(\alpha-1-\big(\mayorchico
{\alpha}{p}-1\big)\big)}-1}{p-1}=\frac{q^{d\big(\alpha-\mayorchico
{\alpha}{p}\big)}-1}{p-1}=v_1(\alpha).
\end{equation}

As a consequence of (\ref{E3.11}), we have 
Proposition \ref{P4.3}.\qed

\s

Proposition \ref{P4.3} proves (\ref{EqNew2}) for $n=1$ and all
$\alpha\in{\ma N}$.

Now consider any cyclic extension $K_n/k$ of degree $p^n$ such that
$P$ is the only ramified prime, it is fully ramified, $\p$
decomposes fully in $K_n/k$ and
${\eu F}_K\mid P^{\alpha}$. We want to prove that $K_n\subseteq
\lam \alpha$, that is, (\ref{EqNew2}): $t_n(\alpha)\leq
v_n(\alpha)$. This will be proved by induction on $n$. The case
$n=1$ is Proposition \ref{P4.3}. We assume that any cyclic extension
$K_{n-1}$ of degree $p^{n-1}$, $n\geq 2$ such that $P$ is the only ramified
prime, $\p$ decomposes fully in $K_{n-1}/k$ and ${\eu F}_{K_{
n-1}}\mid P^{\delta}$
is contained in $\lam \delta$ where $\delta\in{\ma N}$.

Let $K_n$ be any cyclic extension of degree $p^n$ such that
$P$ is the only ramified prime and it is fully ramified,
 $\p$ decomposes fully in $K_n/k$
and ${\eu F}_{K_n}\mid P^{\alpha}$.
Let $K_{n-1}$ be the subfield of $K_n$ of degree $p^{n-1}$
over $k$.
Now we consider $K_n/k$ generated by the Witt vector $\vec{\beta}=
(\beta_1,\ldots,\beta_{n-1},\beta_n)$, that is, $\wp(\vec{y})=
\vec{y}^p\Witt - \vec{y}=\vec{\beta}$,
and we assume that $\vec{\beta}$ is in the normal form 
described by Schmid (see Theorem \ref{T2.3}, \cite{Sch36}).
Then $K_{n-1}/k$ is
given by the Witt vector $\vec{\beta'}=(\beta_1,\ldots,\beta_{n-1})$.

If $\vec{\lambda}:=(\lambda_1,\ldots,\lambda_{n-1},\lambda_n)$ is
the vector of Schmid's parameters, that is, each $\beta_i$
is given by
\begin{gather*}
\beta_i=\frac{Q_i}{P^{\lambda_i}}, \text{\ where\ } Q_i=0 \text{\ 
(that is, $\beta_i=0$) and $\lambda_i=0$ or}\\
 \gcd(Q_i,P)=1, \deg Q_i<\deg P^{\lambda_i}, \lambda_i>0 \text{\ and\ }
\gcd(\lambda_i,p)=1.
\end{gather*}
Since $P$ is fully ramified we have $\lambda_1>0$. 

Now we compute how many different extensions
$K_n/K_{n-1}$ can be constructed
by means of $\beta_n$. 

\begin{lemma}\label{LN1}
For a fixed $K_{n-1}$ the number of different fields $K_n$
is less than or equal to
\begin{equation}\label{EqNew5}
\frac{1+w(\alpha)}{p}=
\frac{1}{p}q^{d(\alpha-\mayorchico {\alpha}p)}.
\end{equation}
\end{lemma}

\proof
For $\beta_n\neq 0$, each equation in normal
form is given by
\begin{gather}\label{E4.11}
y_n^p-y_n=z_{n-1}+\beta_n,
\end{gather}
where $z_{n-1}$ is the element in $K_{n-1}$ obtained by
the Witt generation of $K_{n-1}$ by the vector $\vec{\beta'}$ (see
\cite[page 161]{Sch36}). In fact 
$z_{n-1}$ is given, formally, by
\[
z_{n-1}=\sum_{i=1}^{n-1}\frac{1}{p^{n-i}}\big[
y_i^{p^{n-i}}+\beta_i^{p^{n-i}}-(y_i+\beta_i+
z_{i-1})^{p^{n-i}}\big],
\]
with $z_0=0$.

As in the case $n=1$ we have that there
exist $\Phi(P^{\lambda_n})$
different $\beta_n$ with $\lambda_n>0$.
With the change of variable
$y_n\to y_n+c$, $c\in {\mathcal G}_{\lambda_n}:
=\big\{\frac{h}{P^{\gamma_n}}\mid
h\in R_T, \deg h<\deg P^{\gamma_n}=d\gamma_n
\text{\ or\ } h=0\big\}$
where $\gamma_n=\integer{\lambda_n}{p}$, we obtain
$\beta_n\to \beta_n+\wp(c)$ also in normal form. Therefore
the number of different
elements $\beta_n$ which provide the same field $K_n$ with this
change of variable is $
q^{d(\integerchico{\lambda_n}{p})}$. Therefore we obtain at most
$\Phi\big(P^{\lambda_n-\integerchico{\lambda_n}{p}}\big)$ possible
fields $K_n$ for each $\lambda_n>0$ (see (\ref{E4.8})).
More precisely, if for each $\beta_n$ with $\lambda_n>0$ we set
$\overline{\beta_n}:=\{\beta_n+\wp(c)\mid c\in {\mathcal G}_{\lambda_n}\}$,
then any element of $\overline{\beta_n}$ gives the same field $K_n$.

Let $v_P$ denote the valuation at $P$ and
\begin{gather*}
{\mathcal A}_{\lambda_n}:=\{\overline{\beta_n}\mid v_P(\beta_n)=-
\lambda_n\},\\
{\mathcal A}:=\bigcup_{\substack{\lambda_n=1\\
\gcd(\lambda_n,p)=1}}^{\alpha-1} {\mathcal A}_{\lambda_n}.
\end{gather*}

Then any field $K_n$ is given by $\beta_n=0$ or $\overline{\beta_n}\in
{\mathcal A}$.  From (\ref{E4.8'})
we have that the number of 
fields $K_n$ containing a fixed $K_{n-1}$ 
that we obtain in (\ref{E4.11}) is less than or equal to
\begin{gather}\label{E410'}
1+|{\mathcal A}|=1+w(\alpha)=q^{d\big(\alpha-1-\integerchico
{\alpha-1}{p}\big)}=q^{d\big(\alpha-1-\mayorchico{\alpha}{p}+1\big)}=
q^{d\big(\alpha-\mayorchico{\alpha}{p}\big)}.
\end{gather}

Now with the substitution $y_n\to y_n+jy_1$, $j=0,1,\ldots,p-1$,
in (\ref{E4.11}) we obtain
\[
(y_n+jy_1)^p-(y_n+jy_1)=y_n^p-y_n +j(y_1^p-y_1)=
z_{n-1}+\beta_n+j\beta_1.
\]

Therefore each of the extensions obtained in (\ref{E4.11}) is repeated
at least $p$ times, that is, for each $\beta_n$, we obtain the same extension
with $\beta_n,\beta_n+\beta_1,\ldots, \beta_n+(p-1)\beta_1$. 
We will prove that different $\beta_n+j\beta_1$ correspond to 
different elements of $\{0\}\cup {\mathcal A}$.

Fix $\beta_n$. We modify each $\beta_n+j\beta_1$ into its normal
form: $\beta_n+j\beta_1 +\wp(c_{\beta_n,j})$ for some
$c_{\beta_n,j}\in k$. Indeed $\beta_n+j\beta_1$ is
always in normal form with the possible exception that $\lambda_n=
\lambda_1$ and in this case it holds for at most one index $j\in
\{0,1,\ldots,p-1\}$: if $\lambda_n\neq \lambda_1$, 
\[
v_P(\beta_n+j\beta_1)=
\begin{cases}
-\lambda_n&\text{if $j=0$}\\
-\max\{-\lambda_n,-\lambda_1\}&\text{if $j\neq 0$}
\end{cases}.
\]
When $\lambda_n= \lambda_1$ and if $v_P(\lambda_n+j
\lambda_1)=u>-\lambda_n=-\lambda_1$ and $p|u$,
then for $i\neq j$, $v_P(\beta_n+i\beta_i)=v_P(
\beta_n+j\beta_1+(i-j)\beta_1)=-\lambda_n=-\lambda_1$.
In other words $c_{\beta_n,j}=0$ with very few exceptions.

Each $\mu=\beta_n+j\beta_1+\wp(c_{\beta_n,j})$, $j=0,1,\ldots,p-1$
satisfies that either $\mu=0$ or $\overline{\mu}\in{\mathcal A}$. We will
see that all these elements give different elements of $\{0\}\cup
{\mathcal A}$.

If $\beta_n=0$, then for $j\neq 0$, $v_P(j\beta_1)=-\lambda_1$, so
$\overline{j\beta_n}\in{\mathcal A}$. Now if
$\overline{j\beta_n}=\overline{i\beta_n}$,
then 
\[
j\beta_1=\beta_n^{\prime}+\wp(c_1)\quad\text{and}\quad
i\beta_1=\beta_n^{\prime}+\wp(c_2)
\]
for some $\beta_n^{\prime}\neq 0$ and some
$c_1,c_2\in{\mathcal G}_{\lambda_1}$. It follows that
$(j-i)\beta_1=\wp(c_2-c_1)\in\wp(k)$. This is not possible by the choice
of $\beta_1$ unless $j=i$. 

Let $\beta_n\neq 0$. The case $\beta_n+j\beta_1=0$
for some $j\in\{0,1,\ldots,p-1\}$ has already been considered
in the first case. Thus we consider the case $\beta_n+j\beta_1
+\wp(c_{\beta_n,j})\neq 0$ for all $j$. If for some $i,j\in\{0,1,\ldots,p-1\}$
we have $\overline{\beta_n+j\beta_1+\wp(c_{\beta_n,j})}=\overline{
\beta_n+i\beta_1+\wp(c_{\beta_n,i})}$ then there exists $\beta_n^{\prime}$
and $c_1,c_2\in k$ such that 
\[
\beta_n+j\beta_1+\wp(c_{\beta_n,j})=\beta_n^{\prime}+\wp(c_1)\quad
\text{and}\quad
\beta_n+i\beta_1+\wp(c_{\beta_n,i})=\beta_n^{\prime}+\wp(c_2).
\]
It follows that $(j-i)\beta_1=\wp(c_1-c_2+
c_{\beta_n,i}-c_{\beta_n,j})\in\wp(k)$ so that
$i=j$.

Therefore each field $K_n$ is represented by at least $p$
different elements of $\{0\}\cup {\mathcal A}$. The result follows.
\qed

\s

Now, according to Schmid \cite[page 163]{Sch36}, the conductor of $K_n$
is $P^{M_n+1}$ where $M_n=\max\{pM_{n-1},\lambda_n\}$ and 
$P^{M_{n-1}+1}$ is the conductor of $K_{n-1}$. Since ${\eu F}_{K_n}\mid
P^{\alpha}$, we have $M_n\leq \alpha-1$. Therefore
$pM_{n-1}\leq \alpha-1$ and $\lambda_n\leq \alpha-1$. Hence
${\eu F}_{K_{n-1}}\mid P^{\delta}$ with $\delta=\integer{\alpha-1}{p}+1$.

\begin{proposition}\label{P4.4}
We have
\[
\frac{v_n(\alpha)}{v_{n-1}(\delta)}=
\frac{q^{d\big(\alpha-\mayorchico{\alpha}{p}\big)}}{p},
\]
where $\delta=\integer{\alpha-1}{p}+1$.
\end{proposition}

\proof
From Proposition \ref{P4.1} we obtain
\begin{align*}
v_n(\alpha)&=\frac{q^{d\big(\alpha-\mayorchico{\alpha}
{p^{n-1}}\big)}\big(q^{d\big(\mayorchico{\alpha}{p^{n-1}}
-\mayorchico{\alpha}{p^{n}}\big)}-1\big)}{p^{n-1}(p-1)}\\
&=\frac{q^{d\big(\alpha-\mayorchico{\alpha}
{p^{n-1}}\big)}}{p^{n-1}(p-1)}\big(q^{d\big(\mayorchico{\alpha}{p^{n-1}}
-\mayorchico{\alpha}{p^{n}}\big)}-1\big),\\
\intertext{and}
v_{n-1}(\delta)&=\frac{q^{d\big(\delta-\mayorchico{\delta}
{p^{n-2}}\big)}\big(q^{d\big(\mayorchico{\delta}{p^{n-2}}
-\mayorchico{\delta}{p^{n-1}}\big)}-1\big)}{p^{n-2}(p-1)}\\
&=\frac{q^{d\big(\delta-\mayorchico{\delta}
{p^{n-2}}\big)}}{p^{n-2}(p-1)}\big(q^{d\big(\mayorchico{\delta}{p^{n-2}}
-\mayorchico{\delta}{p^{n-1}}\big)}-1\big).
\end{align*}

From Lemma \ref{L4.2} we have
\begin{align*}
\mayor{\delta}{p^{n-2}}-\mayor{\delta}{p^{n-1}}&=\Big(
\integer{\delta-1}{p^{n-2}}+1\Big)-\Big(\integer{\delta-1}{p^{n-1}}+1\Big)\\
&=\integer{\delta-1}{p^{n-2}}-\integer{\delta-1}{p^{n-1}}=
\integer{\integerchico{\alpha-1}{p}}{p^{n-2}}-
\integer{\integerchico{\alpha-1}{p}}{p^{n-1}}\\
&=\integer{\alpha-1}{p^{n-1}}-\integer{\alpha-1}{p^n}=
\Big(\mayor{\alpha}{p^{n-1}}-1\Big)-\Big(\mayor{\alpha}{p^{n}}-1\Big)\\
&=\mayor{\alpha}{p^{n-1}}-\mayor{\alpha}{p^n},\\
\delta-\mayor{\delta}{p^{n-2}}&=\Big(\integer{\alpha-1}{p}+1\Big)
-\Big(\integer{\delta-1}{p^{n-2}}+1\Big)\\
&=\integer{\alpha-1}{p}-\integer{\delta-1}{p^{n-2}}=\integer{\alpha-1}{p}
-\integer{\integerchico{\alpha-1}{p}}{p^{n-2}}\\
&=\integer{\alpha-1}{p}-\integer{\alpha-1}{p^{n-1}}.
\end{align*}

Therefore
\begin{gather*}
v_{n-1}(\delta)=\frac{q^{d\big(\integerchico{\alpha-1}{p}-\integerchico{
\alpha-1}{p^{n-1}}\big)}}{p^{n-2}(p-1)}\Big(q^{d\big(\mayorchico{\alpha}{p^{n-1}}
-\mayorchico{\alpha}{p^n}\big)}-1\Big).
\end{gather*}
Thus, again by Lemma \ref{L4.2}
\begin{align*}
\frac{v_n(\alpha)}{v_{n-1}(\delta)}&=
\frac{\frac{q^{d\big(\alpha-\mayorchico{\alpha}
{p^{n-1}}\big)}}{p^{n-1}(p-1)}\big(q^{d\big(\mayorchico{\alpha}{p^{n-1}}
-\mayorchico{\alpha}{p^{n}}\big)}-1\big)}
{\frac{q^{d\big(\integerchico{\alpha-1}{p}-\integerchico{
\alpha-1}{p^{n-1}}\big)}}{p^{n-2}(p-1\big)}\Big(q^{d\big(\mayorchico{\alpha}{p^{n-1}}
-\mayorchico{\alpha}{p^n}\big)}-1\Big)}\\
&=\frac{1}{p}q^{d\big(\alpha-\mayorchico{\alpha}{p^{n-1}}-\integerchico{
\alpha-1}{p}+\integerchico{\alpha-1}{p^{n-1}}\big)}\\
&=\frac{1}{p}q^{d\big(\alpha-\mayorchico{\alpha}{p^{n-1}}-\big(\mayorchico{
\alpha}{p}-1\big)+\big(\mayorchico{\alpha}{p^{n-1}}-1\big)\big)}
=\frac{1}{p}q^{d\big(\alpha-\mayorchico{\alpha}{p}\big)}.
\end{align*}

This proves the result. \qed

\s

Hence, from Proposition \ref{P4.4}, Lemma
\ref{LN1} (\ref{EqNew5}) and since 
by the induction hypothesis,
$t_{n-1}(\delta)=v_{n-1}(\delta)$,
we obtain
\[
t_n(\alpha)\leq t_{n-1}(\delta)\big(\frac{1}{p}q^{d\big(\alpha-\mayorchico
{\alpha}p\big)}\big)=v_{n-1}(\delta)\big(\frac{1}{p}q^{d\big(\alpha-\mayorchico
{\alpha}p\big)}\big)=v_n(\alpha).
\]
This proves (\ref{EqNew2}) and Theorem \ref{T2.1}.

\section{Alternative proof of (\ref{EqNew2})}\label{S6}

We keep the same notation as in previous sections. Let $K/k$
be an extension satisfying the conditions (\ref{EqNew1})
and with conductor a divisor of $P^{\alpha}$.
We have ${\eu F}_K=P^{M_n+1}$ where 
\begin{gather*}
M_n=\max\{p^{n-1}\lambda_1,p^{n-2}\lambda_2,\ldots,
p\lambda_{n-1},\lambda_n\},\\
\intertext{see \cite{Sch36}. Therefore}
{\eu F}_K\mid P^{\alpha}\iff M_n+1\leq \alpha \iff
p^{n-i}\lambda_i\leq \alpha-1, \quad i=1,\ldots, n.
\end{gather*}

Thus $\lambda_i\leq \integer{\alpha-1}{p^{n-i}}$. These
conditions give all cyclic extensions of degree $p^n$
where $P\in R_T^+$ is the only ramified prime, it is fully
ramified, $\p$ decomposes fully and its conductor divides
$P^{\alpha}$. Now we estimate the number of different
forms of generating $K$.

Let $K=k(\vec y)$. First, note that with the
change of variable $y_i$ for $y_i+c_i$
for each $i$, $c_i\in k$ we obtain the same extension.
For these new ways of 
generating $K$ to satisfy (\ref{EqNew1}), we
must have:
\l
\item If $\lambda_i=0$, $c_i=0$.
\item If $\lambda_i>0$, then $c_i
\in \Big\{\frac{h}{P^{\gamma_i}}\mid
h\in R_T, \deg h<\deg P^{\gamma_i}=d\gamma_i
\text{\ or\ } h=0\Big\}$,
where $\gamma_i=\integer{\lambda_i}{p}$.
Therefore we have at most 
$\Phi\big(P^{\lambda_i-\integerchico{\lambda_i}{p}}\big)$ 
extensions for this $\lambda_i$ 
(see (\ref{E4.8})). Since $1\leq \lambda_i
\leq\integer{\alpha-1}{p^{n-i}}$ and $\gcd(\lambda_i,p)=1$,
if we let $\delta_i:=\integer{\alpha-1}{p^{n-i}}+1$, 
from (\ref{E3.10}) and (\ref{E4.8'}) we obtain that 
we have at most
\begin{gather}\label{EqNew3}
w(\delta_i)=\sum_{\substack{\lambda_i=1\\ 
\gcd(\lambda_i,p)=1}}^{\delta_i-1}
\Phi\big(P^{\lambda_i-\integerchico{\lambda_i}{p}}\big)=
q^{d\big(\delta_i-1-\integerchico{\delta_i-1}{p}\big)}-1
\end{gather}
different expressions for all possible $\lambda_i>0$.

Now by Lemma \ref{L4.2} we have
\[
\delta_i-1-\integer{\delta_i-1}{p}=\integer{\alpha-1}{p^{n-i}}-
\integer{\integer{\alpha-1}{p^{n-i}}}{p}=\integer{\alpha-1}
{p^{n-i}}-\integer{\alpha-1}{p^{n-i+1}}.
\]
Therefore
\begin{gather}\label{EqNew4}
w(\delta_i)=q^{d\big(\integerchico{\alpha-1}{p^{n-i}}-\integerchico{
\alpha-1}{p^{n-i+1}}\big)}-1.
\end{gather}
\end{list}

When $\lambda_i=0$ is allowed, we have at most $w(\delta_i)+1$
extensions with parameter $\lambda_i$. Therefore, since $\lambda_1
>0$ and $\lambda_i\geq 0$ for $i=2,\ldots, n$, we have that 
the number of extensions satisfying (\ref{EqNew1}) and with
conductor a divisor of $P^{\alpha}$ is at most
\begin{gather*}
s_n(\alpha):=w(\delta_1)\cdot \prod_{i=2}^n\big(w(\delta_i)+1\big).\\
\intertext{From (\ref{EqNew3}) and (\ref{EqNew4}), we obtain}
s_n(\alpha)=\Big(q^{d\big(\integerchico{\alpha-1}{p^{n-1}}
-\integerchico{\alpha-1}{p^n}\big)}-1\Big)\cdot \prod_{i=2}^n
q^{d\big(\integerchico{\alpha-1}{p^{n-i}}
-\integerchico{\alpha-1}{p^{n-i+1}}\big)}.\\
\intertext{Therefore $\prod_{i=2}^n(w(\delta_i)+1)=q^{d\mu}$ where}
\begin{align*}
\mu &= \sum_{i=2}^n \Big(\integer{\alpha-1}{p^{n-i}}
-\integer{\alpha-1}{p^{n-i+1}}\Big)=
\sum_{i=2}^n \integer{\alpha-1}{p^{n-i}}
-\sum_{j=1}^{n-1}\integer{\alpha-1}{p^{n-j}}\\
&=\integer{\alpha-1}{p^{n-n}}-\integer{\alpha-1}{p^{n-1}}=
\alpha-1-\integer{\alpha-1}{p^{n-1}}.
\end{align*}
\end{gather*}

Hence
\begin{align*}
s_n(\alpha)&=\Big(q^{d\big(\integerchico{\alpha-1}{p^{n-1}}
-\integerchico{\alpha-1}{p^n}\big)}-1\Big)\cdot q^{d\big(\alpha
-1-\integerchico{\alpha-1}{p^{n-1}}\big)}\\
&= q^{d\big(\integerchico{\alpha-1}{p^{n-1}}-
\integerchico{\alpha-1}{p^{n}}+\alpha-1-
\integerchico{\alpha-1}{p^{n-1}}\big)}-
q^{d\big(\alpha-1-\integerchico{\alpha-1}{p^{n-1}}\big)}\\
&=q^{d\big(\alpha-1-\integerchico{\alpha-1}{p^{n}}\big)}
-q^{d\big(\alpha-1-\integerchico{\alpha-1}{p^{n-1}}\big)}.
\end{align*}

From Lemma \ref{L4.2} (b) we obtain
\begin{align*}
\alpha-1-\integer{\alpha-1}{p^n}=\alpha-\mayor{\alpha}{p^n}
\quad\text{and}\quad \alpha-1-\integer{\alpha-1}{p^{n-1}}=
\alpha-\mayor{\alpha}{p^{n-1}}.\\
\intertext{Thus}
s_n(\alpha)=q^{\big(\alpha-\mayorchico{\alpha}{p^n}\big)}
-q^{\big(\alpha-\mayorchico{\alpha}{p^{n-1}}\big)}=
p^{n-1}(p-1)v_n(\alpha).
\end{align*}

Finally, the change of variable $\vec y\to \vec j \Witt \times \vec y$ with
$\vec j\in W_n({\ma F}_p)^{\ast}\cong \big({\ma Z}/p^n{\ma Z}
\big)^{\ast}$ gives the same field and we have
$\vec \beta\to \vec j\Witt \times \vec \beta$. Therefore
\[
t_n(\alpha)\leq \frac{s_n(\alpha)}{\varphi(p^n)}=\frac{s_n(\alpha)}
{p^n(p-1)}=v_n(\alpha).
\]
This proves (\ref{EqNew2}) and Theorem \ref{T2.1}.


\begin{thebibliography}{xx}

\bibitem{Has35} Hasse, Helmut, \textit{Theorie der
relativ--zyklischen algebraischen
Funktionen\-k\"orper, insbesondere bei endlichen
Konstantenk\"orper}, J. Reine Angew. Math. \textbf{172},
37--54, (1934).

\bibitem{Hay74} Hayes, David \textit{Explicit Class Field Theory for Rational Function
Fields}, Trans. Amer. Math. Soc. \textbf{189}, 77--91, (1974).

\bibitem{MaRzVi2013} Maldonado--Ram{\'\i}rez, Myriam, Rzedowski--Calder\'on
Martha and Villa--Salvador Gabriel, \textit{Genus fields of abelian extensions
of rational congruence function fields}, Finite Fields and Their Applications
{\bf 20}, 40--54 (2013).

\bibitem{SaRzVi2012} Salas--Torres, Julio Cesar,
Rzedowski--Calder\'on Martha and Villa--Salvador, Gabriel,
\textit{Tamely ramified extensions and cyclotomic fields in characteristic $p$},
Palestine Journal of Mathematics {\bf 2}, No. 1, 1--5 (2013).

\bibitem{SaRzVi2012-2} Salas--Torres, Julio Cesar,
Rzedowski--Calder\'on Martha and Villa--Salvador, Gabriel,
\textit{Artin--Schreier and Cyclotomic Extensions},
arXiv:1306.3716v2.

\bibitem{Sch36} Schmid, Hermann Ludwig, \textit{Zur Arithmetik
der zyklischen p-K\"orper}, J. Reine Angew. Math. 
\textbf{176}, 161--167 (1936).

\bibitem{Vil2006} Villa Salvador, Gabriel Daniel, \textit{Topics in the theory of
algebraic function fields}, Mathematics: Theory \& Applications. Birkh\"auser Boston,
Inc., Boston, MA, 2006.

\end{thebibliography}
\end{document}